\documentclass[11pt,reqno]{amsart}

\usepackage{amsmath,amsthm,amsfonts,amssymb,latexsym,mathrsfs,color}
\usepackage{hyperref}

\newtheorem{theorem}{Theorem}
\newtheorem{corollary}[theorem]{Corollary}
\newtheorem{prop}[theorem]{Proposition}
\newtheorem{conj}[theorem]{Conjecture}

\newcommand{\plat}{{\rm plat\,}}
\newcommand{\fdes}{{\rm fdes\,}}

\newcommand{\ipk}{{\rm ipk\,}}
\newcommand{\asc}{{\rm asc\,}}
\newcommand{\lpk}{{\rm lpk\,}}

\newcommand{\des}{{\rm des\,}}

\newcommand{\msn}{\mathfrak{S}_n}
\newcommand{\mqn}{\mathcal{Q}_n}
\newcommand{\mqnn}{\mathcal{Q}_{n+1}}
\newcommand{\msnn}{\mathfrak{S}_{2n}}
\newcommand{\z}{ \mathbb{Z}}

\newcommand{\altrun}{{\rm altrun\,}}

\newcommand{\Eulerian}[2]{\genfrac{<}{>}{0pt}{}{#1}{#2}}
\newcommand{\Stirling}[2]{\genfrac{\{}{\}}{0pt}{}{#1}{#2}}

\newcommand{\arxiv}[1]{\href{http://arxiv.org/abs/#1}{\texttt{arXiv:#1}}}

\title{Enumeration of a dual set of Stirling permutations by their alternating runs}
\author[S.-M.~Ma]{Shi-Mei~Ma}
\address{School of Mathematics and Statistics,
        Northeastern University at Qinhuangdao,
         Hebei 066004, P.R. China}
\email{shimeimapapers@163.com (S.-M. Ma)}
\author[H.-N.~Wang]{Hai-Na Wang}
\address{Department of Mathematics, Northeastern University, Shenyang, 110004, China}
\email{hainawangpapers@163.com (H.-N. Wang)}

\begin{document}
\maketitle

\begin{abstract}
In this paper, we count a dual set of Stirling permutations by the number of alternating runs.
Properties of the generating functions, including recurrence relations, grammatical interpretations and convolution formulas are studied.
\bigskip

\noindent{\sl Keywords}: Stirling permutations; Alternating runs; Eulerian polynomials
\end{abstract}
\date{\today}

\section{Introduction}
Let $[n]=\{1,2,\ldots,n\}$. The {\it Stirling number of the second kind}
$\Stirling{n}{k}$ is the number of ways to partition $[n]$ into $k$ blocks.
Denote by $D$ the differential operator $\frac{d}{dx}$, and let $\vartheta=xD$. It is well known that
$$\vartheta^n=\sum_{k=1}^n\Stirling{n}{k}z^kD^k.$$
Let $$r(x)=\frac{\sqrt{1+x}}{\sqrt{1-x}}.$$
By induction, one can easily verify that there are positive integers $T(n,k)$, $k\in [2n-1]$, such that
$$\vartheta^n(r(x))=\frac{\sum_{k=1}^{2n-1}T(n,k)x^k}{(1-x)^n(1+x)^{n-1}\sqrt{1-x^2}}\quad\textrm{for $n\ge 1$}.$$
It is clear that the numbers $T(n,k)$ satisfy the initial conditions $T(1,1)=1$ and $T(1,k)=0$ for $k\neq 1$. Let $T_n(x)=\sum_{k=1}^{2n-1}T(n,k)x^k$. Using $\vartheta^{n+1}(r(x))=\vartheta(\vartheta^{n}(r(x)))$,
we get that the polynomials $T_n(x)$ satisfy the recurrence relation
\begin{equation}\label{Tnx-recu}
T_{n+1}(x)=(2nx+1)xT_n(x)+x(1-x^2)T_n'(x)
\end{equation}
for $n\geq 0$, with the initial values $T_0(x)=1$ and $T_1(x)=x$. In particular, $$T_n(1)=-T_{n+1}(-1)=(2n-1)!!\quad\textrm{for $n\ge 1$}.$$
Equating the coefficients of $x^k$ on both sides of~\eqref{Tnx-recu}, we get that the numbers $T(n,k)$ satisfy the recurrence relation
\begin{equation}\label{Tnk-recurrence}
T(n+1,k)=kT(n,k)+T(n,k-1)+(2n-k+2)T(n,k-2).
\end{equation}
The motivating goal of this paper is to find a combinatorial interpretation of the numbers $T(n,k)$.

Let $\msn$ denote the symmetric group of all permutations of $[n]$, where $[n]=\{1,2,\ldots,n\}$.
Let $\pi=\pi(1)\pi(2)\cdots\pi(n)\in\msn$. An {\it interior peak} in $\pi$ is
an index $i\in\{2,3,\ldots,n-1\}$ such that $\pi(i-1)<\pi(i)>\pi(i+1)$.
A {\it left peak} in $\pi$ is an index $i\in[n-1]$ such that $\pi(i-1)<\pi(i)>\pi(i+1)$, where we take $\pi(0)=0$.
Let $\ipk(\pi)$ (resp. $\lpk(\pi)$) be the number of interior peaks (resp. left peaks) in $\pi$.
We say that $\pi$ changes
direction at position $i$ if either $\pi({i-1})<\pi(i)>\pi(i+1)$, or
$\pi(i-1)>\pi(i)<\pi(i+1)$, where $i\in\{2,3,\ldots,n-1\}$. We say that $\pi$ has $k$ {\it alternating
runs} if there are $k-1$ indices $i$ such that $\pi$ changes
direction at these positions. Denote by $\altrun(\pi)$ the number of alternating runs in $\pi$.

Define $$W_n(x)=\sum_{\pi\in\msn}x^{\ipk(\pi)},~\widehat{W}_n(x)=\sum_{\pi\in\msn}x^{\lpk(\pi)},~R_n(x)=\sum_{\pi\in\msn}x^{\altrun(\pi)}.$$
From~\cite[Corollary~2, Theorem 3]{Ma1302}, we get
\begin{equation*}\label{altrun-peak01}
\frac{(1+x)^2}{2x}R_n(x)=xW_n(x^2)+\widehat{W}_n(x^2).
\end{equation*}

Let $R_n(x)=\sum_{k=1}^{n-1}R(n,k)x^k$.
The study of alternating runs of permutations was
initiated by Andr\'e~\cite{Andre84}, and he proved that the numbers $R(n,k)$ satisfy the recurrence relation
\begin{equation}\label{rnk-recurrence}
R(n,k)=kR(n-1,k)+2R(n-1,k-1)+(n-k)R(n-1,k-2)
\end{equation}
for $n,k\ge 1$, where $R(1,0)=1$ and $R(1,k)=0$ for $k\ge 1$.
It follows from~\eqref{rnk-recurrence} that the polynomials $R_n(x)$ satisfy the recurrence relation
\begin{equation}\label{Rnx-recurrence}
R_{n+2}(x)=x(nx+2)R_{n+1}(x)+x\left(1-x^2\right)R_{n+1}'(x),
\end{equation}
with the initial value $R_1(x)=1$.
Recall that a {\it descent} of a permutation $\pi\in\msn$ is a position $i$ such that $\pi(i)>\pi(i+1)$. Denote by $\des(\pi)$ the number of descents of $\pi$. Then the equations
\begin{equation*}
A_n(x)=\sum_{\pi\in\msn}x^{\des(\pi)+1}=\sum_{k=1}^{n}\Eulerian{n}{k}x^{k},
\end{equation*}
define the {\it Eulerian polynomial} $A_n(x)$ and the {\it Eulerian number} $\Eulerian{n}{k}$.
The polynomial $R_n(x)$ is closely related to $A_n(x)$:
\begin{equation}\label{rnx-anx}
R_n(x)=\left(\dfrac{1+x}{2}\right)^{n-1}(1+w)^{n+1}A_n\left(\dfrac{1-w}{1+w}\right),\quad
w=\sqrt{\frac{1-x}{1+x}},
\end{equation}
which was first established by David and
Barton~\cite[157-162]{DB62} and then stated more concisely by
Knuth~\cite[p.~605]{Knuth73}.
There is a large literature devoted to the polynomials $R_n(x)$~(see~\cite[A059427]{Sloane}). The reader is referred to~\cite{CW08,Ma1302} for recent results on this subject.

In~\cite{Carlitz},
Carlitz introduced $C_n(x)$ defined by
$$\sum_{n=0}^\infty \Stirling{n+k}{k}x^n=\frac{C_n(x)}{(1-x)^{2k+1}},$$
and asked for a combinatorial interpretation of $C_n(x)$.
Riordan~\cite{Riordan76} noted that $C_n(x)$ is the enumerator of trapezoidal words with n elements
by number of distinct elements, where trapezoidal words are such that the $i$-th element takes the
values $1,2,\ldots,2i-1$.
Gessel and Stanley~\cite{Gessel78} gave another combinatorial interpretation of $C_n(x)$ in terms of descents of Stirling permutations.
A {\it Stirling permutation} of order $n$ is a permutation $\sigma=\sigma(1)\sigma(2)\cdots\sigma(2n-1)\sigma(2n)$ of the multiset $\{1,1,2,2,\ldots,n,n\}$ such that
for each $i$, $1\leq i\leq n$, all entries between the two occurrencies of $i$ are larger than $i$.
Denote by $\mqn$ the set of Stirling permutation of order $n$.
For $\sigma\in\mqn$, let $\sigma(0)=\sigma(2n+1)=0$, and
let
\begin{align*}
  \des(\sigma)&=\#\{i\mid\sigma(i)>\sigma(i+1)\}, \\
  \asc(\sigma)&=\#\{i\mid\sigma(i-1)<\sigma(i)\}, \\
  \plat(\sigma)&=\#\{i\mid\sigma(i)=\sigma(i+1)\}
\end{align*}
denote the number of descents, ascents and plateaux of $\sigma$, respectively. Gessel and Stanley~\cite{Gessel78} proved that
$$C_n(x)=\sum_{\sigma\in\mqn}x^{\des{\sigma}}.$$
B\'ona~\cite[Theorem~1]{Bona08} introduced the plateau statistic on $\mqn$, and proved that descents, ascents and plateaux are
equidistributed over $\mqn$. The reader is referred to~\cite{Haglund12,Janson11,Kuba12} for recent
progress on the study of statistics on Stirling permutations.

In the next section, we show that $T_n(x)$ is the enumerator of a dual set of Stirling permutations of order $n$
by number of alternating runs.
\section{Combinatorial interpretation of $T(n,k)$}\label{Section-2}
Let $\sigma=\sigma(1)\sigma(2)\cdots\sigma({2n})\in\mqn$.
Let $\Phi$ be the injection which maps each first occurrence of entry $j$ in $\sigma$ to $2j$ and the
second $j$ to $2j-1$,
where $j\in [n]$. For example, $\Phi(221331)=432651$.
The {\it dual set} $\Phi(\mqn)$ of $\mqn$ is defined by $$\Phi(\mqn)=\{\pi\mid \sigma\in\mqn, \Phi(\sigma)=\pi\}.$$
Clearly, $\Phi(\mqn)$ is a subset of $\msnn$.
For $\pi\in\Phi(\mqn)$, the entry $2j$ is to the left of $2j-1$, and all entries in $\pi$ between $2j$ and $2j-1$ are larger than $2j$, where $1\leq j\leq n$.
Let $ab$ be an ascent in $\sigma$, so $a<b$.
Using $\Phi$, we see that $ab$ maps into $(2a-1)(2b-1)$, $(2a-1)(2b)$, $(2a)(2b-1)$ or $(2a)(2b)$, and vice versa.
Note that $\asc(\sigma)=\asc(\Phi(\sigma))=\asc(\pi)$.
Therefore, we have
$$C_n(x)=\sum_{\pi\in\Phi(\mqn)}x^{\asc(\pi)}.$$

It should be noted that $\pi\in \Phi(\mqn)$ always ends with a descending run. We now present the following result.
\begin{theorem}
We have
$$T(n,k)=\#\{\pi\in\Phi(\mqn)\mid \altrun(\pi)=k\}.$$
\end{theorem}
\begin{proof}
There are three ways in which a permutation $\pi\in\Phi(\mqnn)$ with $\altrun(\pi)=k$
can be obtained from a permutation $\sigma\in\Phi(\mqn)$ by inserting the pair (2n+2)(2n+1) into consecutive positions.
\begin{enumerate}
\item [(a)] If $\altrun(\sigma)=k$, then we can insert the pair $(2n+2)(2n+1)$ right before the beginning of each descending run, and right after the end of
each ascending run. This accounts for $kT(n,k)$ possibilities.
\item [(b)] If $\altrun(\sigma)=k-1$, then we distinguish two cases: when $\sigma$ starts in an
ascending run, we insert the pair $(2n+2)(2n+1)$ to the front of $\sigma$; when $\sigma$ starts in an
descending run, we insert the pair $(2n+2)(2n+1)$ right after the first entry of $\sigma$.
This gives $T(n,k-1)$ possibilities.
\item [(c)] If $\altrun(\sigma)=k-2$,
then we can insert the pair $(2n+2)(2n+1)$ into the remaining $(2n+1)-(k-2)-1=2n-k+2$ positions.
This gives $(2n-k+2)T(n,k-2)$ possibilities.
\end{enumerate}
Therefore, the numbers $T(n,k)$ satisfy the recurrence relation~\eqref{Tnk-recurrence}, and this completes the proof.
\end{proof}

Define $$M_n(x)=\sum_{\pi\in\Phi(\mqn)}x^{\ipk(\pi)},~N_n(x)=\sum_{\pi\in\Phi(\mqn)}x^{\lpk(\pi)}.$$
It follows from~\cite[Theorem~4]{Ma15} that $M_n(x)=x^nN_n\left(\frac{1}{x}\right)$. Moreover, from~\cite[Theorem~5]{Ma15}, we have
\begin{equation*}\label{symmetric}
(1+x)T_n(x)=xM_n(x^2)+N_n(x^2).
\end{equation*}

We now recall some properties of $N_n(x)$.
Let $N_n(x)=\sum_{k=1}^nN(n,k)x^k$. Apart from counting permutations in the set $\Phi(\mqn)$ with $k$ left peaks, the number $N(n,k)$ also has the following
combinatorial interpretations:
\begin{enumerate}
   \item [\rm ($m_1$)]  Let $e=(e_1,e_2,\ldots,e_n)\in\z^n$, and let
$I_{n,k}=\left\{ e\in \z^n|0\leq e_i\leq (i-1)k\right\}$, which known as the set of $n$-dimensional {\it $k$-inversion sequences} (see~\cite{Savage1201}).
The number of {\it ascents} of $e$ is defined by
$$\asc(e)=\#\left\{i:1\leq i\leq n-1\big{|}\frac{e_i}{(i-1)k+1}<\frac{e_{i+1}}{ik+1}\right\}.$$ Savage and Viswanathan~\cite{Savage1202} discovered that
$N(n,k)=\#\{e\in I_{n,2}: \asc(e)=n-k\}$.
  \item [\rm ($m_2$)] We say that an index $i\in [2n-1]$ is an {\it ascent plateau} of $\pi\in\mqn$ if $\pi(i-1)<\pi(i)=\pi(i+1)$.
  The number $N(n,k)$ counts Stirling permutations in $\mqn$ with $k$ ascent plateaux~(see~\cite[Theorem~3]{Ma15}).
  \item [\rm ($m_3$)] The number $N(n,k)$ counts perfect matching on $[2n]$ with the restriction that only $k$ matching pairs with odd minimal elements (see~\cite{Ma1502}).
\end{enumerate}
The polynomials $N_n(x)$ satisfy the recurrence relation
$$N_{n+1}(x)=(2n+1)xN_n(x)+2x(1-x)N'_n(x)$$
with initial value $N_0(x)=1$. The first few of $N_n(x)$ are
$$N_1(x)=x,
N_2(x)=2x+x^2,
N_3(x)=4x+10x^2+x^3,
N_4(x)=8x+60x^2+36x^3+x^4.$$
The exponential generating function for $N_n(x)$ is given as follows (see~\cite[Section~5]{Ma13}):
\begin{equation}\label{EGF-Nnx}
N(x,z)=\sum_{n\geq 0}N_n(x)\frac{z^n}{n!}=\sqrt{\frac{1-x}{1-xe^{2z(1-x)}}}.
\end{equation}

A polynomial $f(x)=\sum_{k=0}^na_kx^k$ is {\em symmetric} if $a_k=a_{n-k}$ for all $0\leq k\leq n$, while it is {\em unimodal} if there exists an index
$0\leq m\leq n$, such that $$a_0\leq a_1\leq \cdots\leq a_{m-1}\leq a_m\geq a_{m+1}\geq \cdots \geq a_n.$$
\begin{theorem}
The polynomial $T_n(x)$ is symmetric and unimodal.
\end{theorem}
\begin{proof}
It is immediate from~\eqref{symmetric} that $T_n(x)$ is a symmetric polynomial. We show the unimodality by induction on $n$. Note that $T_1(x)=x, T_2(x)=x+x^2+x^3$ and $T_3(x)=x+3x^2+7x^3+3x^4+x^5$ are all unimodal. Thus it suffices to consider the case $n\geq 3$. Assume that $T_n(x)$ is symmetric and unimodal. For $1\leq k\leq n+1$, it follows from~\eqref{Tnk-recurrence} that
\begin{align*}
 T(n+1,k)-T(n+1,k-1)& =(k-1)(T(n,k)-T(n,k-1))+(T(n,k-1)-T(n,k-2))\\
  & +(2n-k+2)(T(n,k-2)-T(n,k-3))+(T(n,k)-T(n,k-3)) \\
  &\geq 0,
\end{align*}
where the inequalities are follow from the induction hypothesis. This completes the proof.
\end{proof}

In the next section, we present a grammatical interpretation of $T_n(x)$.
\section{Grammatical interpretations}\label{Section-3}
The grammatical method was introduced by Chen~\cite{Chen93}
in the study of exponential structures in combinatorics. For an alphabet $A$, let $\mathbb{Q}[[A]]$ be the rational commutative ring of formal power
series in monomials formed from letters in $A$. A context-free grammar over
A is a function $G: A\rightarrow \mathbb{Q}[[A]]$ that replace a letter in $A$ by a formal function over $A$.
The formal derivative $D$ is a linear operator defined with respect to a context-free grammar $G$. More precisely,
the derivative $D=D_G$: $\mathbb{Q}[[A]]\rightarrow \mathbb{Q}[[A]]$ is defined as follows:
for $x\in A$, we have $D(x)=G(x)$; for a monomial $u$ in $\mathbb{Q}[[A]]$, $D(u)$ is defined so that $D$ is a derivation,
and for a general element $q\in\mathbb{Q}[[A]]$, $D(q)$ is defined by linearity.

The {\it hyperoctahedral group} $B_n$ is the group of signed permutations of the set $\pm[n]$ such that $\pi(-i)=-\pi(i)$ for all $i$, where $\pm[n]=\{\pm1,\pm2,\ldots,\pm n\}$.
For each $\pi\in B_n$,
we define
\begin{equation*}
\begin{split}
\des_A(\pi)&:=\#\{i\in\{1,2,\ldots,n-1\}|\pi(i)>\pi({i+1})\},\\
\des_B(\pi)&:=\#\{i\in\{0,1,2,\ldots,n-1\}|\pi(i)>\pi({i+1})\},
\end{split}
\end{equation*}
where $\pi(0)=0$.
Following~\cite{Adin01}, the {\it flag descent number} of $\pi$
is defined by
\begin{equation*}
\fdes(\pi):=\begin{cases}
2\des_A(\pi)+1,& \text{if $\pi(1)<0$};\\
2\des_A(\pi), & \text{otherwise}.
\end{cases}
\end{equation*}

Let
\begin{equation*}
\begin{split}
B_n(x)&=\sum_{\pi\in B_n}x^{\des_B(\pi)}=\sum_{k=0}^nB(n,k)x^{k},\\
S_n(x)&=\sum_{\pi\in B_n}x^{\fdes(\pi)}=\sum_{k=1}^{2n}S(n,k)x^{k-1}.
\end{split}
\end{equation*}
The polynomial $B_n(x)$ is called an {\it Eulerian polynomial of type $B$}, while $B(n,k)$ is called an {\it Eulerian number of type $B$} (see~\cite[A060187]{Sloane}).
It follows from~\cite[Theorem 4.3]{Adin01} that the numbers $S(n,k)$ satisfy the recurrence relation
\begin{equation*}\label{Snk}
S(n,k)=kS(n-1,k)+S(n-1,k-1)+(2n-k+1)S(n-1,k-2)
\end{equation*}
for $n,k\geq 1$, where $S(1,1)=S(1,2)=1$ and $S(1,k)=0$ for $k\geq 3$. The polynomials $S_n(x)$ is closely related to the Eulerian polynomial $A_n(x)$:
$$S_n(x)=\frac{1}{x}(1+x)^nA_n(x)\quad\textrm{for $n\ge 1$},$$
which was established by Adin, Brenti and Roichman~\cite{Adin01}.
It should be noted that $S_n(x)$ and $A_n(x)$ are both symmetric.

Consider the context-free grammar
\begin{equation*}\label{Context:01}
A=\{x,y,z\},~G=\{x\rightarrow p(x,y,z), y\rightarrow q(x,y,z), z\rightarrow r(x,y,z)\},
\end{equation*}
where $p(x,y,z),q(x,y,z)$ and $r(x,y,z)$ are polynomials in $x,y$ and $z$.
The {\it diamond product} of $z$ with the grammar $G$ is defined by
$$G\diamond z=\{x\rightarrow p(x,y,z)z, y\rightarrow q(x,y,z)z, z\rightarrow r(x,y,z)z\}.$$

We now recall two results on
context-free grammars.
\begin{prop}[{\cite[Theorem~6]{Ma1302}}]
If
\begin{equation}\label{Context:02}
G=\{x\rightarrow xy, y\rightarrow yz,z\rightarrow y^2\},
\end{equation}
then
\begin{equation*}\label{grammatical-alternating}
D^{n}(x^2)=x^2\sum_{k=0}^{n}R(n+1,k)y^kz^{n-k}.
\end{equation*}
Setting $x=z=1$, we have $D^{n}(x^2)|_{x=z=1}=R_{n+1}(y)$.
\end{prop}

\begin{prop}[{\cite[Theorem~10]{Ma1303}}]\label{Ma1303}
Consider the context-free grammar
\begin{equation}\label{Context:03}
G'=\{x\rightarrow xyz, y\rightarrow yz^2, z\rightarrow y^2z\},
\end{equation}
which is the diamond product of $z$ with the grammar $G$ defined by~\eqref{Context:02}.
For $n\geq 1$, we have
\begin{equation*}
\begin{split}
D^n(xy)&=x\sum_{k=1}^{2n}S(n,k)y^{2n-k+1}z^{k},\\
D^n(yz)&=\sum_{k=0}^nB(n,k)y^{2n-2k+1}z^{2k+1},\\
D^n(y)&=\sum_{k=1}^nN(n,k)y^{2n-2k+1}z^{2k},\\
D^n(z)&=\sum_{k=1}^nN(n,n-k+1)y^{2n-2k+2}z^{2k-1},\\
D^n(y^2)&=2^n\sum_{k=1}^{n}\Eulerian{n}{k}y^{2n-2k+2}z^{2k}.
\end{split}
\end{equation*}
\end{prop}

We can now conclude the following result.
\begin{theorem}\label{mthm2}
Let $G'$ be the context-free grammar given by~\eqref{Context:03}.
Then for $n\geq 1$, we have
\begin{equation*}
\begin{split}
D^n(x)&=x\sum_{k=1}^{2n-1}T(n,k)y^kz^{2n-k},\\
D^n(x^2)&=2x^2(y+z)^{n-1}\sum_{k=1}^{n}\Eulerian{n}{k}y^kz^{n-k+1}.
\end{split}
\end{equation*}
Setting $x=z=1$, we have $D^{n}(x)|_{x=z=1}=T_{n}(y)$ and  $D^{n}(x^2)|_{x=z=1}=2(1+y)^{n-1}A_{n}(y)$.
\end{theorem}
\begin{proof}
Note that $D(x)=xyz$ and $D^2(x)=xyz^3+xy^2z^2+xy^3z$.
For $n\geq 1$,
we define $t(n,k)$ by
\begin{equation*}\label{Dnx-def}
D^n(x)=x\sum_{k\geq 1}t(n,k)y^kz^{2n-k}.
\end{equation*}
Then
\begin{align*}
  D^{n+1}(x)& =D(D^n(x)) \\
              & =x\sum_{k\geq1}t(n,k)y^{k+1}z^{2n-k+1}+x\sum_{k\geq1}kt(n,k)y^{k}z^{2n-k+2}+x\sum_{k\geq1}(2n-k)t(n,k)y^{k+2}z^{2n-k}.
\end{align*}
Hence
\begin{equation}\label{tnk-recu2}
t(n+1,k)=kt(n,k)+t(n,k-1)+(2n-k+2)t(n,k-2).
\end{equation}
By comparing~\eqref{tnk-recu2} with~\eqref{Tnk-recurrence}, we see that
the numbers $t(n,k)$ satisfy the same recurrence relation and initial conditions as $T(n,k)$, so they agree.
The assertion for $D^n(x^2)$ can be proved in a similar way.
\end{proof}

It follows from {\it Leibniz's formula} that
\begin{equation*}\label{Dnab-Leib}
D^n(uv)=\sum_{k=0}^n\binom{n}{k}D^k(u)D^{n-k}(v).
\end{equation*}
Hence
$$D^n(x^2)=\sum_{k=0}^n\binom{n}{k}D^k(x)D^{n-k}(x),$$
$$D^{n+1}(x)=D^n(xyz)=\sum_{k=0}^n\binom{n}{k}D^k(x)D^{n-k}(yz)=\sum_{k=0}^n\binom{n}{k}D^k(xy)D^{n-k}(z).$$
Therefore, we can use Proposition~\ref{Ma1303} and Theorem~\ref{mthm2} to get several convolution identities.
\begin{corollary}\label{cor1}
For $n\geq 1$, we have
\begin{equation}\label{convolution}
2(1+x)^{n-1}A_n(x)=\sum_{k=0}^n\binom{n}{k}T_k(x)T_{n-k}(x),
\end{equation}
$$T_{n+1}(x)=x\sum_{k=0}^n\binom{n}{k}T_k(x)B_{n-k}(x^2),$$
$$T_{n+1}(x)=x\sum_{k=0}^n\binom{n}{k}S_k(x)N_{n-k}(x^2).$$
\end{corollary}

Let $T(x,z)=\sum_{n=0}^\infty T_n(x)\frac{z^n}{n!}$.
Recall that the exponential generating function for $A_n(x)$ is given as follows (see~\cite[A008292]{Sloane}):
\begin{equation}\label{Axz}
A(x,t)=\sum_{n\geq 0}A_n(x)\frac{t^n}{n!}=\frac{1-x}{1-xe^{t(1-x)}}.
\end{equation}
Combining~\eqref{convolution} and~\eqref{Axz}, we get
\begin{equation}\label{Txz-EGF}
T(x,z)=\frac{{e^{z \left( x-1 \right)  \left( x+1 \right) }}+x}{1+x}\sqrt {{\frac {1-x^2}{{e^{2\,z \left( x-1
\right) \left( x+1 \right) }}-x^2}}}.
\end{equation}
From~\eqref{EGF-Nnx}, we have
$$\sum_{n\geq 0}M_n(x^2)\frac{z^n}{n!}=\sum_{n\geq 0}x^{2n}N_n\left(\frac{1}{x^2}\right)\frac{z^n}{n!}=\sqrt {{\frac {1-x^2}{{e^{2\,z \left( x-1
\right) \left( x+1 \right) }}-x^2}}}.$$
Note that $$\frac{{e^{z \left( x-1 \right)  \left( x+1 \right) }}+x}{1+x}=1+\sum_{n\geq 1}(x-1)^n(x+1)^{n-1}\frac{z^n}{n!}.$$
Therefore, from~\eqref{Txz-EGF}, we obtain
$$T_n(x)=M_n(x^2)+\sum_{k=0}^{n-1}\binom{n}{k}M_k(x^2)(x-1)^{n-k}(x+1)^{n-k-1}\quad\textrm{for $n\ge 1$}.$$

\section{Concluding remarks}
In this paper, we show that the polynomial $T_n(x)$ has many similar properties to $R_n(x)$. In fact, there are more similar properties deserve to be studied.
From the relation~\eqref{rnx-anx} and
the fact that $A_n(x)$ have only real zeros, Wilf~\cite{Wilf} proved that
$R_n(x)$ have only real zeros for $n\ge 2$. Moreover, it follows from~\eqref{Rnx-recurrence} that
all zeros of $R_n(x)$ belong to $[-1,0]$, and the zeros of $R_n(x)$ separate that of $R_{n+1}(x)$ (see~\cite[Corollary~8]{Ma2008}).

Let $f(x)$ and $F(x)$ be two polynomials with only real coefficients. Suppose that $f(x)$ and $F(x)$ both have only imaginary zeros. We say that $f(x)$ {\it separates} $F(x)$ if $\deg F=\deg f+2$ and
the sequences of real and imaginary parts of the zeros of $f(x)$ respectively separate that of $F(x)$.  In other words, let $f(x)=a\prod_{j=1}^{n-1}(x+p_j+q_j\mathrm i)(x+p_j-q_j\mathrm i)$, and let $F(x)=b\prod_{j=1}^{n}(x+s_j+t_j\mathrm i)(x+s_j-t_j\mathrm i)$, where $a,b$ are respectively leading coefficients of $f(x)$ and $F(x)$, $p_1\geq p_2\geq \cdots\geq p_{n-1},~q_1\geq q_2\geq \cdots\geq q_{n-1},s_1\geq s_2\geq \cdots\geq s_n$ and $t_1\geq t_2\geq \cdots\geq t_n$. Then we have $$s_1\geq p_1\geq s_2\geq p_2\geq \cdots\geq s_{n-1}\geq p_{n-1}\geq s_n,$$
$$t_1\geq q_1\geq t_2\geq q_2\geq \cdots\geq t_{n-1}\geq q_{n-1}\geq t_n.$$

Based on empirical evidence, we propose the following conjecture.
 \begin{conj}\label{conj}
For $n\geq 2$, all zeros of $T_n(x)/x$ are imaginary and $T_n(x)/x$ separates $T_{n+1}(x)/x$.
 \end{conj}


\end{document}